\documentclass[11pt, a4paper, oneside, reqno]{amsart}

\usepackage[utf8]{inputenc}
\usepackage{amsmath, amsthm, amssymb, amsfonts}
\usepackage{bbm}
\usepackage[foot]{amsaddr}
\usepackage{booktabs}
\usepackage{cite}
\usepackage[shortlabels]{enumitem}
\usepackage[colorlinks=true, linkcolor = blue, citecolor = dgreen]{hyperref}
\usepackage{cleveref}
\usepackage[numbers]{natbib}
\usepackage{mathrsfs}
\usepackage[dvipsnames]{xcolor}
\usepackage{tikz}
\usepackage{geometry}
\usepackage{setspace}
\usepackage[normalem]{ulem}
\usepackage{newpxtext}

\setlist{leftmargin=*, topsep=0.5em, parsep=0pt, itemsep=1em, labelindent=0pt, align=left}

\definecolor{red}{HTML}{D62728}
\definecolor{blue}{RGB}{ 0, 109, 219}
\definecolor{dgreen}{rgb}{0,.8,0}

\crefname{enumi}{}{}
\Crefname{enumi}{}{}

\crefformat{enumi}{(#2#1#3)}
\Crefformat{enumi}{(#2#1#3)}
\crefrangeformat{enumi}{(#3#1#4) to~(#5#2#6)}
\Crefrangeformat{enumi}{(#3#1#4) to~(#5#2#6)}

\makeatletter
\newcommand{\ie}{\emph{i.e.}\@ifnextchar.{\!\@gobble}{}}
\newcommand{\eg}{\emph{e.g.}\@ifnextchar.{\!\@gobble}{}}
\newcommand{\etc}{etc\@ifnextchar.{}{.\@}}
\makeatother

\newtheorem*{theorem*}{Theorem}
\newtheorem{thm}{Theorem}[section]

\newtheorem{lemma}[thm]{Lemma}
\newtheorem{prop}[thm]{Proposition}
\newtheorem{corollary}[thm]{Corollary}

\newtheorem{remark}[thm]{Remark}

\crefname{lemma}{Lemma}{Lemmas}
\numberwithin{equation}{section}
\def\P{{\mathbb P}} %Probability
\def\Q{{\mathbb Q}} %Probability
\def\E{{\mathbb E}} % Esperance

\def\Var{\mathrm{Var}}
\def \N{{\mathbb{N}}}
\def \R {{\mathbb{R}}}

\newcommand{\ep}{\epsilon}

\newcommand{\sgn}{\text{sgn}}

\begin{document}
\title[]{On the exponential integrability of the derivative of intersection and self-intersection local time for Brownian motion and related processes
}
\author{Kaustav Das$^{\dagger \ddagger}$}
\author{Gregory Markowsky$^{\dagger}$}
\author{Binghao Wu$^{\dagger}$}
\address{$^\dagger$School of Mathematics, Monash University, Victoria, 3800 Australia.}
\address{$^\ddagger$Centre for Quantitative Finance and Investment Strategies, Monash University, Victoria, 3800 Australia.}
\email{kaustav.das@monash.edu, greg.markowsky@monash.edu, binghao.wu@monash.edu}
\date{}
\maketitle

\begin{abstract}
We show that the derivative of the intersection and self-intersection local times of alpha-stable processes are exponentially integrable for certain parameter values. This includes the Brownian motion case. We also discuss related results present in the literature for fractional Brownian motion, and in particular give a counter-example to a result in [Guo, J., Hu, Y., and Xiao, Y., Higher-order derivative of intersection local time for two independent fractional Brownian motions, Journal of Theoretical Probability 32, (2019), pp. 1190-1201] related to this question.
\end{abstract}

%-----------------------------------------------------------Introduction--------------------------------------------------------
%--------------------------------------------------------------------------------------------------------------------------------

\section{Introduction}
\label{sec:introduction}
\noindent
%------------ILT and SLT
Let $B^1$ and $B^2$ be independent real valued Brownian motions. The intersection local-time (ILT) of $B^1$ and $B^2$ is formally defined as
\begin{align}
\label{eqn:ILT}
	\int_0^T \int_0^T \delta(B^1_t  - B^2_s) ds dt 
\end{align}
where $\delta$ is the Dirac delta function. Intuitively, ILT measures the amount of time the processes $B^1$ and $B^2$ spend intersecting each other on the time interval $[0, T]$. Similarly, let $B$ be a real valued Brownian motion. The self-intersection local-time (SLT) of $B$ is formally defined as 
\begin{align}
\label{eqn:SLT}
	\int_0^T \int_0^t \delta(B_t  - B_s) ds dt.
\end{align}
Intuitively, SLT measures the amount of time the process $B$ spends revisiting prior attained values on the time interval $[0, T]$. 

Consider the following functional introduced in \citep{rogers1991intrinsic, rogers1991local, rogers1991t},
\begin{align*}
    A(T,B_{T})=\int_{0}^{T}1_{[0,\infty)}(B_T-B_s)ds.
\end{align*}
A formal application of It\^o's formula yields the formula:
\begin{align*}
    \frac{1}{2}\int_{0}^T\int_{0}^t\delta'(B_{t}-B_{s})dsdt + \frac{1}{2} \sgn(x)t =
    \int_0^t L_s^{B_s - x} dB_s - \int_{0}^{t} \sgn(B_{t}-B_{u} - x) du. 
\end{align*} 
A slightly different formula was stated as a formal identity without proof in \citep{rosen2005derivatives}, and this formula was rigorously proved in \citep{markowsky2008proof}. We note in particular the random variable
\begin{align} 
	\int_{0}^T\int_{0}^t\delta'(B_{t}-B_{s})dsdt, \label{eqn:DSLTBM}
\end{align}
which is referred to as the derivative of self-intersection local-time (DSLT) of $B$. This is the principle object of study in this paper.

As \cref{eqn:ILT}, \cref{eqn:SLT}, and \cref{eqn:DSLTBM} are formal expressions, a first step in giving precise meaning to them is by approximating $\delta$ with the Gaussian heat kernel
\begin{align}
\label{eqn:gaussian}
	\rho_{\epsilon}(x) :=\frac{1}{\sqrt{2\pi \epsilon}}e^{-\frac{x^{2}}{2\epsilon}} 
\end{align}
which we note converges weakly to $\delta$ as $\ep \downarrow 0$. It is important to note that in the study of ILT and SLT, it is more convenient to utilise the representation of $\rho_{\epsilon}$ through the Fourier transform
\begin{align}
\label{eqn:gaussianfourier}
	\rho_{\epsilon}(x) = \frac{1}{2\pi} \int_{\R}e^{ipx}e^{-\frac{p^{2}\epsilon}{2}}dp. 
\end{align}
Letting $\gamma_\epsilon$ denote either \cref{eqn:ILT} or \cref{eqn:SLT} with $\delta$ replaced with $\rho_\epsilon$, then providing precise meaning to ILT or SLT amounts to showing that $\gamma_\epsilon$ converges as $\epsilon \downarrow 0$ in some manner. Indeed, this has been the topic of various articles, see for example \citep{bass2004self, berman1969local, berman1973local, chen2004large, le1994exponential}. We also mention the definitive resource on local times of Markov processes, \citep{marcus2006markov}. Additionally, it would be remiss to not mention the comprehensive work of \citep{geman1980occupation} for an extensive review on the general theory of local time (LT), which in particular introduces the notion of considering LT, ILT and SLT as occupation densities. However, for this article we will not require this interpretation. 

%------------DILT and DSLT
Recently there has been a surge of interest in DSLT, as well as the derivatives of ILT, which we denote by DILT; see for instance
\citep{yan2017derivative, yan2022derivative, guo2019higher}.  In a similar manner to before, in order to make sense of \cref{eqn:DSLTBM}, one approximates $\delta$ by $p_\ep$, then differentiates its Fourier representation \cref{eqn:gaussianfourier} to obtain the expression
\begin{align*}
	\frac{i}{2 \pi} \int_{0}^T \int_{0}^{t} \int_{\mathbb{R}} p e^{ip (B_t - B_s)}  e^{-\frac{p^2}{2 \epsilon}} dpdsdt,
\end{align*}
which is then shown to converge as $\epsilon\downarrow 0$ a.s. and in $L^p$ for all $p > 0$. DSLT is then defined to be this limit. Note that derivatives of order higher than one do not exist for Brownian motion (although see \citep{newguy}) but do for fractional Brownian motion with certain values of the Hurst parameter $H$; see \citep{geman1984local, yu2021higher, yu2023smoothness, das2022existence, kuang2022derivative}. Recently, DSLT has been considered for higher-order intersections, in \citep{guo2023derivative}, following an initial work in \citep{rosen2010continuous}.

%------------Applications
We will say that a random variable $X$ is {\it exponentially integrable} of order $\beta$ if there exists a constant $M > 0$ such that $\E[\exp \{ M|X|^{\beta} \}]<\infty$. Exponential integrability in general is vital for many subfields of probability theory, since with $\beta=1$ it is equivalent to the existence of the moment generating function $M_X(t) := \E[e^{tX}]$ for $t$ in a neighborhood of $0$, and can also be used to give strong tail estimates on the distribution of $X$. The relationship of this concept to ILT, SLT, and their variants is due primarily to the applications of these processes in the physical sciences to study various phenonema. For example, if we let $\gamma$ denote the SLT of Brownian motion and noting that SLT provides a measure of self intersection, one can define a probability measure
\begin{align*}
	\Q(d \omega) = C\exp\{M \gamma^\beta \} \P(d \omega)
\end{align*}
where $\P$ denotes the standard Wiener measure and $M$ and $\beta$ are constants, whereas $C$ is a normalising constant. The probability measure $\Q$ then provides a model of self attracting or self avoiding Brownian motion depending on the sign of $M$ ($> 0$ and $< 0$ respectively). Of course, whether or not $\Q$ is well-defined hinges on whether $C\exp\{M \gamma^\beta\}$ is indeed a Radon-Nikodym derivative, and exponential integrability provides an affirmative answer to this question. Such motivation (and other motivations) is discussed in \citep{bass2004self,konig2006brownian,le1994exponential}, among other places.  

%-------Kaustav is writing about alpha stable processes here
In this article, we will consider the exponential integrability of DILT and DSLT of symmetric stable processes with index of stability $\alpha \in (0, 2]$, also known simply as symmetric $\alpha$-stable processes. For the purposes of this article, it is enough to consider stochastic processes taking values in $\R$. In the following it is understood that $\alpha \in (0, 2]$. We recall that an $\alpha$-stable process is a real valued L\'evy process $X = (X_t)_{t \geq 0}$ with the property that $X_1$ is a strictly stable random variable with index of stability $\alpha$, meaning that for any $A, B > 0$ and independent copies $X_1^{(1)}, X_1^{(2)}$ of $X_1$, we have $A X^{(1)}_1 + B X^{(2)}_1 \stackrel{d}{=} (A^{\alpha} + B^{\alpha})^{1/\alpha} X_1$. Moreover, it is well known that a L\'evy process is an $\alpha$-stable process if and only if it possesses the self similarity property $X_{at} \stackrel{d}{=} a^{1/\alpha} X_t$ for any $a > 0$, where $\alpha$ denotes the index of stability of $X_1$. A symmetric $\alpha$-stable process is an $\alpha$-stable process where $X_1$ is a strictly stable symmetric random variable. In this case, the characteristic function of $X_t$ admits the convenient form 
\begin{align*}
    \E[e^{iu X_t}] = e^{-t \sigma^{\alpha}|u|^{\alpha}}
\end{align*}
for some $\sigma > 0$. For further insights we refer to \citep{lee1992dependence, applebaum2009levy, samoradnitsky2017stable}.

Exponential integrability of DILT has been studied in the literature before, most notably in \citep{guo2019higher, zhou2023derivatives}.
In this article we are primarily interested in DSLT, rather than DILT, due to the applications given above, as well as its more intricate structure, not to mention the fact that this question seems not to have been addressed in the literature. However, the first step is to study the question for DILT, which we do in \Cref{sec:mainresultsILT}. The result given there is then used to deduce the desired result for DSLT using a method pioneered by Le Gall, which we do in \Cref{sec:mainresultsSLT} (and we discuss Le Gall's method in detail in \Cref{appen:LeGall}). In \Cref{sec:DSLTfBM} we discuss the case for fractional Brownian motion. We also include two appendices at the end containing required technical facts. 

\begin{remark}
    In the rest of the article, we will be content with considering the objects we study over the time interval $[0, 1]$ without loss of generality, as the case of $[0, T]$ can be obtained trivially via scaling.
\end{remark}

%---------------------------------------------------Exp DILT---------------------------------------------------------------------
%--------------------------------------------------------------------------------------------------------------------------------

\section{Exponential integrability for derivative of intersection local-time}
\label{sec:mainresultsILT}
\noindent
In this section we state and prove results regarding the exponential integrability of DILT for symmetric $\alpha$-stable processes and Brownian motion.

Let $X^1, X^2$ be independent symmetric stable processes with the same index of stability $\alpha$. Consider the DILT of $X^1$ and $X^2$, which can be expressed as
\begin{align*}
    \theta:=\frac{i}{2\pi}\int_{0}^{1}\int_{0}^{1}\int_{\R}pe^{ip(X_{t}^{1}-X_{s}^{2})}dpdsdt.
\end{align*}
The existence of $\theta$ is proved in Rosen \citep{rosen2005derivatives}. 
We prove the following regarding the exponential integrability of $\theta$.
%-------------Theorem: Exponential integrability DSLTsym
\begin{thm}
\label{DILT exponential integrability}
Suppose the common index of stability $\alpha \in (\frac{3}{2}, 2]$ and let $\beta \in [0, \frac{\alpha}{3})$. Then there exists a constant $M > 0$ such that $\E[\exp \{ M|\theta|^{\beta} \}]<\infty$.
\end{thm}
%-------------Proof
\begin{remark}
 There is overlap between this result and Theorem 1.1 in the recent paper \citep{zhou2023derivatives}. However, we are interested in a simpler situation than is considered there, since our primary interest is DSLT (which is not considered in that paper). We therefore include a proof of this result which is significantly simpler than that given in \citep{zhou2023derivatives} for the benefit of the reader.
\end{remark}

\begin{proof}
Since $|\theta|>0$, it suffices to focus on the $n$-th moment of $|\theta|$ and then use the Maclaurin series for the exponential function. We first proceed by considering the case of even $n$. We have
\begin{align*}
	\E[|\theta|^{n}]&=\E[\theta^{n}] \\
        &=\E \left [\frac{i^{n}}{(2\pi)^{n}}\int_{[0,1]^{2n}}\int_{\R^{n}}\prod_{j=1}^{n}p_{j}e^{\sum_{j=1}^{n}ip_{j}(X_{t_{j}}^{1}-X_{s_{j}}^{2})}dpdsdt \right ] \\  
        &\leq \frac{1}{(2\pi)^{n}}\int_{[0,1]^{2n}}\int_{\R^{n}}\prod_{j = 1}^{n}p_{j}\E \left [e^{\sum_{j=1}^{n}ip_{j}(X_{t_{j}}^{1}-X_{s_{j}}^{2})} \right]dpdsdt.
\end{align*}
Let $\Delta=\{(t_{0},\dots, s_{\sigma(n)})|0=t_{0}<t_{1}<\cdots<t_{n}<1,0=s_{\sigma(0)}<s_{\sigma(1)}<\cdots<s_{\sigma(n)}<1\}$, and denote by $\Phi_{n}$ the set of all permutations of $\{1,\dots,n\}$. Write
\begin{align*}
    u_{j} &= t_{j}-t_{j-1}, 
    &u^{\prime}_{j} &= s_{\sigma(j)}-s_{\sigma(j-1)}, \\
    v_{j} &= \sum_{k=j}^{n}p_{j},
    &v^{\prime}_{j} &= \sum_{k=j}^{n}p_{\sigma(j)}.
\end{align*}

Due to the independence between $X^{1}$ and $X^{2}$ we obtain
\begin{align*}
        	\E[|\theta|^{n}] &\leq \frac{n!}{(2\pi)^{n}}\sum_{\sigma\in \Phi_{n}}\int_{\Delta}\int_{\R^{n}}\prod_{j=1}^{n}p_{j}\E\left[e^{\left(i\sum_{j=1}^{n}v_{j}\left(X_{t_{j}}^{1}-X^{1}_{t_{j-1}}\right)\right)}\right]\E\left[e^{\left(i\sum_{j=1}^{n}v^{'}_{\sigma(j)}\left(X_{s_{\sigma(j)}}^{2}-X_{s_{\sigma(j-1)}}^{2}\right)\right)}\right]dpdsdt \\
        	&\leq \frac{n!}{(2\pi)^{n}}\sum_{\sigma\in \Phi_{n}}\int_{\Delta}\int_{\R^n}\prod_{j=1}^{n}p_{j}e^{-\sum_{j=1}^{n}u_{j}v_{j}^{\alpha}}e^{-\sum_{j=1}^{n}u^{\prime}_{j}v_{j}^{\prime \alpha}}dpdsdt.
\end{align*}
Combining the fact that $e^{-cx^{\alpha}}<\frac{1}{1+cx^{\alpha}}$ when $c>0$ yields
\begin{align*}
      \E[|\theta|^{n}]\leq \frac{n!}{(2\pi)^{n}}\sum_{\sigma\in \Phi_{n}}\int_{\Delta}\int_{\R^n}\prod_{j=1}^{n}|p_{j}|\frac{1}{1+u_{j}v_{j}^{\alpha}}\frac{1}{1+u^{\prime}_{j}v_{j}^{\prime \alpha}}dpdsdt.
\end{align*}
Noting that $\prod_{j=1}^{n}|p_{j}|<\prod_{j=1}^{n}v_{j}^{2}+|v_{j}|+1$ since $p_{j}=v_{j}-v_{j-1}$, and utilising the Cauchy Schwarz inequality we obtain
\begin{align*}
        \E[|\theta|^{n}]&\leq \frac{n!}{(2\pi)^{n}}\sum_{\sigma\in \Phi_{n}}\int_{\Delta}\left(\int_{\R^n}\prod_{j=1}^{n}\frac{(v_{j}^{2}+|v_{j}|+1)}{(1+u_{j}v_{j}^{\alpha})^{2}}dp\right)^{\frac{1}{2}}\left(\int_{\R^n}\prod_{j=1}^{n}\frac{(v_{\sigma(j)}^{2}+|v_{\sigma(j)}|+1)}{(1+u^{\prime}_{j}v_{j}^{\prime \alpha})^{2}}dp\right)^{\frac{1}{2}}dsdt \\
        &\leq\frac{n!}{(2\pi)^{n}}\sum_{\sigma\in \Phi_{n}}\int_{\Delta}\left(\int_{\R^n}\prod_{j=1}^{n}\frac{(v_{j}^{2}+|v_{j}|+1)u^{\frac{3}{2}}_{j}}{(1+u_{j}|v_{j}|^{\alpha})^{2}u^{\frac{3}{2}}_{j}}dp\right)^{\frac{1}{2}}\left(\int_{\R^n}\prod_{j=1}^{n}\frac{(v_{\sigma(j)}^{2}+|v_{\sigma(j)}|+1)u^{\prime \frac{3}{2}}_{j}}{(1+u^{\prime}_{j}|v_{j}|^{\prime \alpha})^{2}u^{\prime\frac{3}{2}}_{j}}dp\right)^{\frac{1}{2}}dsdt.
\end{align*}
When $\frac{3}{2}<\alpha<2$, then we have $\int_{\R}\frac{v^{2}u^{\frac{3}{\alpha}}}{(1+u|v|^{\alpha})^{2}} dv = C <\infty$. Thus
\begin{align}
       	\E[|\theta|^{n}]&\leq \frac{n!}{(2\pi)^{n}}\sum_{\sigma\in \Phi_{n}}\int_{\Delta}C^{n}\prod_{j}^{n}u_{j}^{-\frac{3}{2\alpha}}u_{j}^{\prime-\frac{3}{2\alpha}}dsdt \nonumber \\
        &\leq \frac{(n!)^{2}C^{n}}{(2\pi)^{n}}\int_{\Delta}\prod_{j}^{n}u_{j}^{-\frac{3}{2\alpha}}u_{j}^{\prime-\frac{3}{2\alpha}}dsdt. \label{eqn:upperbound}
\end{align}
By \Cref{Beta function} we can upper bound the integral in \cref{eqn:upperbound} to obtain
\begin{align}
\begin{split}
\label{eq:moments for stable processes}
       \E[|\theta|^{n}]& \leq \frac{C^{n}(n!)^{2}}{(2\pi)^{n}\Gamma(n(1-\frac{3}{2\alpha})+1)^{2}} \\
        & \leq (n!)^{\frac{3}{\alpha}}C_{even}^{n},
\end{split}
\end{align}
where the second inequality is obtained via \Cref{lmag}. This handles the even moment case. 

The odd moment case can be tackled by combining the result on the even moments (\cref{eq:moments for stable processes}) with Jensen's inequality. Assuming $n$ is odd, and utilising Jensen's inequality, we obtain
\begin{align}
\begin{split}
    	\E[|\theta|^{n}]&=\E \left [|\theta|^{n\frac{n+1}{n+1}}\right]  \\
    	&\leq \E\left[|\theta|^{n+1} \right]^{\frac{n}{n+1}} \\
    	&\leq C^{n}_{even}(\left(n+1)!\right)^{\frac{3n}{\alpha(n+1)}} \\
    	&\leq C^{n}_{even}(n!)^{\frac{3}{\alpha}}\left(\frac{n+1}{n}\right)^{\frac{(n+1)3}{\alpha}} \\
    	&\leq C^{n}_{odd}(n!)^{\frac{3}{\alpha}},
\end{split}
\end{align}
where we have obtained the preceding 2nd inequality via \cref{eq:moments for stable processes}. Regarding the 3rd inequality, we have applied \Cref{lmag} to $((n+1)!)^{\frac{n}{n+1}}(\frac{n}{n+1})^{n+1}$:
\begin{align*}
    ((n+1)!)^{\frac{n}{n+1}}\left (\frac{n}{n+1}\right )^{n+1}&\leq \Gamma\left (\left (\frac{n}{n+1}\right )(n+1)+1 \right)
\end{align*}
and thus
\begin{align*}
    ((n+1)!)^{\frac{n}{n+1}}&\leq (n!)\left (\frac{n+1}{n}\right )^{n+1}.
\end{align*}
as claimed. Hence for $0\leq\beta<\frac{\alpha}{3}$ and $n\in\N$ we have
\begin{align}
\label{constant for stable}
    \E\left[|\theta|^{\beta n}\right]\leq \E\left[|\theta|^{n}\right]^\beta\leq K^{n\beta}(n!)^{\frac{3\beta}{\alpha}},
\end{align}
where $K=\max(C_{even},C_{odd})$. Hence,
\begin{align*}  
	\E\left[e^{M|\theta|^{\beta}}\right]&=\sum_{n=0}^{\infty}\frac{M^{n}\E[|\theta|^{\beta n}]}{n!} \\
    	&\leq \sum_{n=0}^{\infty}M^{n}K^{n}(n!)^{\frac{3\beta}{\alpha}-1}< \infty.
\end{align*}
\end{proof}

We repeat that this includes the Brownian motion case, and isolate it as a corollary.

\begin{corollary}
\label{DILT expoential integrability for Brownian motion}
Let $\alpha = 2$ (the Brownian motion case) and $\beta \in [0, \frac{2}{3})$. Then there exists a constant $M > 0$ such that $\E[\exp \{ M|\theta|^{\beta} \}]<\infty$.
\end{corollary}

%------------------------------------------------------------Exp DSLT-----------------------------------------------------------
%------------------------------------------------------------------------------------------------------------------------------

\section{Exponential integrability for derivative of self-intersection local-time}
\label{sec:mainresultsSLT}
\noindent
In this section we state and prove results regarding the exponential integrability of DSLT for symmetric $\alpha$-stable processes and Brownian motion.

Let $X$ be a symmetric $\alpha$-stable process. Consider its DSLT, which can be expressed as
\begin{align*}
	\hat{\theta} = \frac{i}{2\pi}\int_{0}^{1}\int_{0}^{t}\int_{\R}pe^{ip(X_{t}-X_{s})}dpdsdt,
\end{align*}
whose existence is proven by Rosen \citep{rosen2005derivatives}.

%-------------Theorem: Exponential integrability DSLT Sym of order alpha
\begin{thm}
\label{thm:DSLTsymalpha}
 Suppose the index of stability for $X$ is $\alpha \in (\frac{4}{3}, 2]$ and let $\gamma \in [0, \frac{2\alpha}{6+\alpha})$. Then there exists a constant $M > 0$ such that $\E[\exp \{ M|\hat{\theta}|^{\gamma} \}]<\infty$.    
\end{thm}

\begin{proof}
Before we proceed with the proof, we will utilise the scheme of Le Gall from \citep{le1994exponential} which we briefly describe in \Cref{appen:LeGall}, in order to rewrite $\hat{\theta}$ as:
\begin{align*}
    \hat{\theta}&=\lim_{N\to \infty}\sum_{n=1}^{N}\sum_{k=1}^{2^{n-1}}\hat{\theta}_{n,k}, \\
    \hat{\theta}_{n,k}&:= \int_{2^{-n}(2k-2)}^{2^{-n}(2k-1)}\int_{2^{-n}(2k-1)}^{2^{-n}2k}\int_{\R}\frac{i}{2\pi}pe^{ip(X_{t}-X_{s})}dpdsdt.
\end{align*}
By the self-similarity property of symmetric $\alpha$-stable processes,
\begin{align}
\label{stable scaling}
     \hat{\theta}_{n,k}\stackrel{d}=2^{2n(\frac{1}{\alpha}-1)}\theta,
\end{align}
   where $\theta$ is the DILT of two independent symmetric stable processes with the same index of stability $\alpha$. Note that for each fixed $n \in \N$, the $\hat \theta_{n,k}$ are mutually independent for $k = 1, \dots, 2^{n-1}$. To begin with we can construct $b_{N}=\prod_{j=2}^{N}(1-2^{-a(j-1)})$ with $0<a<1$. Let us define $M:=\lim_{N\to \infty}b_{N}$. We then consider the following expectation:
\begin{align*}
   \E\left[\exp\left\{b_{N}\left|\sum_{n=1}^{N}\sum_{k=1}^{2^{n-1}}\hat{\theta}_{n,k}\right|^{\gamma}\right\}\right].
\end{align*}
Using \Cref{inequality beta}, we obtain
        \begin{align*}
          \E\left[\exp\left\{b_{N}\left|\sum_{n=1}^{N}\sum_{k=1}^{2^{n-1}}\hat{\theta}_{n,k}\right|^{\gamma}\right\}\right]&\leq \E\left[\exp\left\{b_{N}\left|\sum_{n=1}^{N-1}\sum_{k=1}^{2^{n-1}}\hat{\theta}_{n,k}\right|^{\gamma}+b_{N}\left|\sum_{k=1}^{2^{N-1}}\hat{\theta}_{N,k}\right|^{\gamma}\right\}\right].
        \end{align*}
H\"older's inequality yields
        \begin{align*}
         \E\left[\exp\left\{b_{N}\left|\sum_{n=1}^{N}\sum_{k=1}^{2^{n-1}}\hat{\theta}_{n,k}\right|^{\gamma}\right\}\right]&\leq\E\left[\exp\left\{b_{N-1}\left|\sum_{n=1}^{N-1}\sum_{k=1}^{2^{n-1}}\hat{\theta}_{n,k}\right|^{\gamma}\right\}\right]^{1-2^{-a(N-1)}}\\
         &\times \E\left[\exp\left\{b_{N}2^{a(N-1)}\left|\sum_{k=1}^{2^{N-1}}\hat{\theta}_{N,k}\right|^{\gamma}\right\}\right]^{2^{-a(N-1)}}.
        \end{align*}
Since $\E[\exp{|X|}]>1$ for any random variable $X$,  we can upper bound the quantity by dropping the index $1-2^{-a(N-1)}$,
        \begin{align*}
          \E\left[\exp\left\{b_{N}\left|\sum_{n=1}^{N}\sum_{k=1}^{2^{n-1}}\hat{\theta}_{n,k}\right|^{\gamma}\right\}\right]&\leq \E\left[\exp\left\{b_{N-1}\left|\sum_{n=1}^{N-1}\sum_{k=1}^{2^{n-1}}\hat{\theta}_{n,k}\right|^{\gamma}\right\}\right]\\
          &\times\E\left[\exp\left\{b_{N}2^{a(N-1)}\left|\sum_{k=1}^{2^{N-1}}\hat{\theta}_{N,k}\right|^{\gamma}\right\}\right]^{2^{-a(N-1)}}.
        \end{align*}
By the self-similarity property of the $\alpha$-stable process and the fact that each $\hat{\theta}_{N,k}$ is independent and has the same distribution as $2^{2N(\frac{1}{\alpha}-1)}\theta$, we may use $\theta_{k}$ to represent the independent copies of $\theta$. We can then rewrite the upper bound as follows
        \begin{align*}
          \E\left[\exp\left\{b_{N}\left|\sum_{n=1}^{N}\sum_{k=1}^{2^{n-1}}\hat{\theta}_{n,k}\right|^{\gamma}\right\}\right]&\leq\E\left[\exp\left\{b_{N-1}\left|\sum_{n=1}^{N-1}\sum_{k=1}^{2^{n-1}}\hat{\theta}_{n,k}\right|^{\gamma}\right\}\right]\\
          &\times\E\left[\exp\left\{b_{N}2^{a(N-1)}2^{2N\left(\frac{1}{\alpha}-1\right)\gamma}\left|\sum_{k=1}^{2^{N-1}}\theta_{k}\right|^{\gamma}\right\}\right]^{2^{-a(N-1)}}.
        \end{align*}
By the monotone convergence theorem, we obtain
        \begin{align*}
         \E\left[\exp\left\{b_{N}\left|\sum_{n=1}^{N}\sum_{k=1}^{2^{n-1}}\hat{\theta}_{n,k}\right|^{\gamma}\right\}\right]&\leq\E\left[\exp\left\{b_{N-1}\left|\sum_{n=1}^{N-1}\sum_{k=1}^{2^{n-1}}\hat{\theta}_{n,k}\right|^{\gamma}\right\}\right]\\
         &\times\left(\sum_{l=0}^{\infty}\frac{b_{N}^{l}2^{alN-al+\frac{2Nl\gamma}{\alpha}-2Nl\gamma}}{l!}\E\left[\left|\sum_{k=1}^{2^{N-1}}\theta_{k}\right|^{l\gamma}\right]\right)^{2^{-a(N-1)}}.
        \end{align*}
Utilising Jensen's inequality, we obtain 
     \begin{align*}
         \E\left[\exp\left\{b_{N}\left|\sum_{n=1}^{N}\sum_{k=1}^{2^{n-1}}\hat{\theta}_{n,k}\right|^{\gamma}\right\}\right]&\leq\E\left[\exp\left\{b_{N-1}\left|\sum_{n=1}^{N-1}\sum_{k=1}^{2^{n-1}}\hat{\theta}_{n,k}\right|^{\gamma}\right\}\right]\\
         &\times\left(\sum_{l=0}^{\infty}\frac{b_{N}^{l}2^{alN-al+\frac{2Nl\gamma}{\alpha}-2Nl\gamma}}{l!}\E\left[\left|\sum_{k=1}^{2^{N-1}}\theta_{k}\right|^{l}\right]^{\gamma}\right)^{2^{-a(N-1)}}.
        \end{align*}
According to \Cref{Moments inequality authentic} and \cref{constant for stable}, we can upper bound the expectation of a sum of independent random variables by
\begin{align*}
\E\left[\exp\left\{b_{N}\left|\sum_{n=1}^{N}\sum_{k=1}^{2^{n-1}}\hat{\theta}_{n,k}\right|^{\gamma}\right\}\right]&\leq\E\left[\exp\left\{b_{N-1}\left|\sum_{n=1}^{N-1}\sum_{k=1}^{2^{n-1}}\hat{\theta}_{n,k}\right|^{\gamma}\right\}\right]\\
         &\times\left(\sum_{l=0}^{\infty}K^{l\gamma}b_{N}^{l}2^{Nl(a+\frac{2\gamma}{\alpha}-\frac{3\gamma}{2})-al+\gamma+2l\gamma}(l!)^{\frac{3\gamma}{\alpha}+\frac{\gamma}{2}-1}\right)^{2^{-a(N-1)}},
\end{align*}
where $K$ is the same constant as in \cref{constant for stable}. 
When $\frac{4}{3}<\alpha\leq 2$, $0<\gamma<\frac{2\alpha}{\alpha+6}$, we can choose $a\in(0,1)$ such that the sum $$\sum_{l=0}^{\infty}K^{l\gamma}b_{N}^{l}2^{Nl(a+\frac{2\gamma}{\alpha}-\frac{3\gamma}{2})-al+\gamma+2l\gamma}(l!)^{\frac{3\gamma}{\alpha}+\frac{\gamma}{2}-1}$$ is bounded for all $N$. Let us denote one of the bounds as $C$. Then
\begin{align*}
    \E\left[\exp\left\{b_{N}\left|\sum_{n=1}^{N}\sum_{k=1}^{2^{n-1}}\hat{\theta}_{n,k}\right|^{\gamma}\right\}\right]\leq&\E\left[\exp\left\{b_{N-1}\left|\sum_{n=1}^{N-1}\sum_{k=1}^{2^{n-1}}\hat{\theta}_{n,k}\right|^{\gamma}\right\}\right]C^{2^{-a(N-1)}}\\
    &\leq \prod_{j=1}^{N}C^{2^{-a(j-1)}}.
\end{align*}
Therefore according to Fatou's Lemma, we obtain
\begin{align*}
    \E\left[\exp\left\{M\left|\hat{\theta}\right|^{\gamma}\right\}\right]&\leq \liminf_{N\to \infty}\E\left[\exp\left\{M\bigg|\sum_{n=1}^{N}\sum_{k=1}^{2^{n-1}}\hat{\theta}_{n,k}\bigg|^{\gamma}\right\}\right]\\
    &\leq \lim_{N\to \infty}\prod_{j=1}^{N}C^{2^{-a(j-1)}}< \infty.
\end{align*}
It is well known that $\lim_{N\to \infty}\prod_{j=1}^{N}C^{2^{-a(j-1)}}$ converges if and only if $\lim_{N\to\infty}\sum_{j=1}^{N}2^{-a(j-1)}\ln C$ converges. Therefore the upper bound is finite when $a$ is positive. This completes the proof. 
\end{proof}
%-------------Theorem: Exponential integrability for DSLTBM
We reiterate that the preceding theorem includes the Brownian motion case, and isolate it as a corollary.
\begin{corollary}
Let $\alpha = 2$ (the Brownian motion case) and $\gamma \in [0, \frac{1}{2})$. Then there exists a constant $M > 0$  such that $\E[\exp\{M|\hat{\theta}|^{\gamma}\}]<\infty$.
\end{corollary}

%----------------------------------------------Discussion about FBMs-------------------------------------------------------------
%--------------------------------------------------------------------------------------------------------------------------------

\section{On the DILT of fractional Brownian motion}
\label{sec:DSLTfBM}
\noindent
We would like to extend our result to the case of fractional Brownian motion (fBm), however the lack of independent increments makes applying Le Gall's scheme directly problematic. It is therefore difficult to address the DSLT of fBm. The DILT is likely to be more tractable, albeit still difficult, as we now explain.

The property of local nondeterminism is generally used in place of independence when working with fractional Brownian motion. In this context, this property asserts that \citep{berman1973local},
\begin{align*}
\Var\left(\sum_{k=1}^{n}a_{k}(B_{t_{k}}^{H}-B_{t_{k-1}}^{H})\right)\geq c_{n,H}\sum_{k=1}^{n}a^{2}_{k}\Var\left(B_{t_{k}}^{H}-B_{t_{k-1}}^{H}\right)=c_{n,H}\sum_{k=1}^{n}a^{2}_{k}(t_{k}-t_{k-1})^{2H},
\end{align*}
where $c_{n,H}$ depends on $n$ and $H$. This is enough to show finiteness of all moments in certain cases, and existence of the process in $L^p(\Omega)$, but is not enough by itself when trying to prove exponential integrability, as we need to bound all moments simultaneously, and this requires strict knowledge of the constant $c_{n,H}$.

The paper \citep{guo2019higher} is devoted to DILT of fBm, and claims a result on exponential integrability, however we were unable to follow some of the arguments there, and ultimately found a counterexample to one of its results. Theorem 1 in that paper is claimed as follows.

\begin{theorem*}[{\citep[Theorem ~1]{guo2019higher}}]
    Let $B^{H_{1}}$ and $W^{H_{2}}$ be two independent d-dimensional fractional Brownian motions of Hurst parameter $H_{1}$ and $H_{2}$, respectively.
    \begin{enumerate}[label = (\roman*), ref = \roman*]
        \item \label{thm:Hu1} Assume $k=(k_{1},...,k_{d})$ is an index of nonnegative integers (meaning that $k_{1},...k_{d}$ are nongegative integers) satisfying
        \begin{align}
            \label{HU's condition}
            \frac{H_{1}H_{2}}{H_{1}+H_{2}}(|k|+d)<1,
        \end{align}
        where $|k|=k_{1}+\cdots+k_{d}.$ Then, the $k$-th order derivative intersection local time $L_{k,d}$ exists in $L^{p}(\Omega)$ for any $p\in [1,\infty)$, where
        $$L_{k,d,T}:=\frac{i^{|k|}}{(2\pi)^{d}}\int_{0}^{T}\int_{0}^{T}\int_{\R^{d}}\prod_{j=1}^{d}p^{k_{j}}_{j}e^{ip(B_{t}^{H_{1}}-W_{s}^{H_{2}})}dpdsdt.$$
        \item \label{thm:Hu2} Assume \cref{HU's condition} is satisfied. There is a strictly positive constant $C_{d,k,T}\in (0,\infty)$ such that
        $$\E\left[e^{C_{d,k,T}|L_{k,d}|^{\beta}}\right]<\infty,$$
        where $\beta=\frac{H_{1}+H_{2}}{2dH_{1}H_{2}}$.
%        \item[iii)] If $L_{k,d,T}\in \mathcal{L}^{1}(\Omega),$ where $k=(0,...,0,k_{i},0,...,0)$ with $k_{i}$ being even integer, the condition  \cref{HU's condition} must be satisfied.
    \end{enumerate}
\end{theorem*}

If we choose $T=1$, $d=1$, $k=2$, $H_{1}=H_{2}=\frac{1}{2}$ which satisfies the condition \cref{HU's condition}, then according to this result we should have exponential integrability. However, we will be able to show that $\E[L_{2,1,1}^{2}]=\infty$, and this contradicts $\cref{thm:Hu1}$ and clearly precludes the exponential integrability of this process. Writing $B^{1/2} \equiv B$ and $W^{1/2} \equiv W$ we have
\begin{align*}
    \E[L_{2,1,1}^2]&=\frac{1}{4\pi^{2}}\int_{[0,1]^{4}}\int_{\R^{2}}p_{1}^{2}p_{2}^{2}\E\left[e^{ip_{1}(B_{t_{1}}-W_{s_{1}})+ip_{2}(B_{t_{2}}-W_{s_{2}})}\right]dpdsdt.
\end{align*}
Since the integrand is positive and $D_{t}=\{t_{1},t_{2}:0<t_{1}<\frac{1}{2},0<t_2-t_{1}<\frac{1}{2}\}\subset [0,1]^{2}$, so is $D_{s} = \{s_{1},s_{2}:0<s_{1}<\frac{1}{2},0<s_{2}-s_{1}<\frac{1}{2}\}$. Thus
\begin{align*}
    \E\left[L_{2,1,1}^{2}\right]&\geq \frac{1}{4\pi^{2}}\int_{D_{t}}\int_{D_{s}}\int_{\R^{2}}p_{1}^{2}p_{2}^{2}e^{-\frac{1}{2}(t_{2}-t_{1})p_{2}^{2}-\frac{1}{2}t_{1}(p_{1}+p_{2})^{2}}e^{-\frac{1}{2}(s_{2}-s_{1})p_{2}^{2}-\frac{1}{2}s_{1}(p_{1}+p_{2})^{2}}dpdsdt\\
    &=\frac{1}{4\pi^{2}}\int_{\R^{2}}p_{1}^{2}p_{2}^{2}K(p_{2})^{2}K(p_{1}+p_{2})^{2}dp,
\end{align*}
where 
\begin{align*}
K(x)=\int_{0}^{\frac{1}{2}}e^{-\frac{1}{2}x^{2}t}dt=
\begin{cases}
   \frac{1}{2}, &x=0,\\
    \frac{1-e^{-\frac{1}{4}x^{2}}}{\frac{1}{2}x^{2}}, &\text{otherwise}.
\end{cases} 
\end{align*}
By the construction of $K(x)$, we know there exists a number $\lambda>0$, such that when $|x|<\lambda$, $K(x)>\frac{1}{4}$. We also can find positive constants $c_{1}$ and ${c_{2}}$ such that $\frac{c_{1}}{1+x^{2}}<K(x)<\frac{c_{2}}{1+x^{2}}.$ Therefore,
\begin{align*}
    \E\left[L_{2,1,1}^{2}\right]&\geq \frac{1}{4\pi^{2}}\int_{\R}p_{2}^{2}K(p_{2})^{2}dp_{2}\int_{|p_{1}+p_{2}|<\lambda}p_{1}^{2}K(p_{1}+p_{2})^{2}dp_{1}\\
    &\geq  \frac{1}{64\pi^{2}}\int_{\R}p_{2}^{2}K(p_{2})^{2}dp_{2}\int_{|p_{1}+p_{2}|<\lambda}p_{1}^{2}dp_{1}\\
    &\geq \frac{C}{64\pi^{2}}\int_{\R}p_{2}^{4}K(p_{2})^{2}dp_{2}\\
    &\geq \frac{Cc_{2}^{2}}{64\pi^{2}}\int_{\R}\frac{p_{2}^{4}}{(1+p_{2}^{2})^{2}}dp_{2}=\infty.
\end{align*}
Therefore, $L_{2,1,1}$ does not exist in $L^{2}(\Omega)$. 

We remark that the method used in \citep{guo2019higher} is sophisticated, and it is to be hoped that the methods established there can be repaired in order to recover the correct result.
%-------------------------------------------------------Appendix-----------------------------------------------------------------
%--------------------------------------------------------------------------------------------------------------------------------
%--------------------------------------------------------------------------------------------------------------------------------

\appendix

%------------------------------------------------Miscellaneous properties--------------------------------------------------------
%--------------------------------------------------------------------------------------------------------------------------------

\section{Miscellaneous properties}
\label{appen:misc}
\noindent
%-------------Defn: Gamma function
In this section we collect some of the technical estimates and facts which were used in the proofs of the theorems. Many of these facts can be found elsewhere, but we include proofs of most of them for the benefit of the reader.

The standard Gamma function is defined as follows:
\begin{align*}
	\Gamma(x)=\int_{0}^{\infty}t^{x-1}e^{-t}dt.
\end{align*}

This function is well defined except for negative integers, and satisfies $x\Gamma(x) = \Gamma(x+1)$. In this article we will only need to utilise the Gamma function with positive arguments. We require the following fact. 

\begin{lemma}
    The Gamma function is logarithmically convex; that is, $\ln\Gamma(x)$ is convex.
\end{lemma}

\begin{proof}
    Let $f(x)=\ln\Gamma(x)$, then
\begin{align*}
    f^{''}(x)&=\frac{\Gamma^{''}(x)}{\Gamma(x)}-\frac{(\Gamma^{'}(x))^{2}}{(\Gamma(x))^{2}} \\
    &=\frac{\int_{0}^{\infty}t^{x-1}e^{-t}dt\int_{0}^{\infty}(\ln t)^{2}t^{x-1}e^{-t}dt-(\int_{0}^{\infty}(\ln t)t^{x-1}e^{-t}dt)^{2}}{(\Gamma(x))^{2}}.
\end{align*}
Applying H\"older's inequality to the last line will show that $f^{''}(x)>0$. Consequently, $f(x)$ is convex.
\end{proof}

%-------------Prop: Gautschi's inequality
\begin{prop}[Gautschi's inequality]
    For any $s\in(0,1)$ and $x\geq 0$,
    $$x^{1-s}\Gamma(x+s)\leq \Gamma(x+1)\leq \Gamma(x+s)(x+s)^{1-s}.$$
\end{prop}
%-------------Proof
\begin{proof}
Since the Gamma function is logarithmically convex, for $0<s<1$,
\begin{align*}
    \Gamma(x+s)=\Gamma(x(1-s)+(x+1)s))&\leq \Gamma(x)^{(1-s)}\Gamma(x+1)^{s} \\
    &= x^{s-1}(x\Gamma(x))^{(1-s)}\Gamma(x+1)^s \\
    &= x^{s-1}\Gamma(x+1).
\end{align*}   

This proves the first inequality. The second follows similarly:
\begin{align*}
    \Gamma(x+1) = \Gamma((x+s)s+(x+s+1)(1-s))&\leq \Gamma(x+s)^{s}\Gamma(x+s+1)^{1-s} \\
    &= \Gamma(x+s)^s((x+s)\Gamma(x+s))^{1-s} \\
    &= (x+s)^{1-s}\Gamma(x+s).
\end{align*}
\end{proof}

%-------------Lemma
\begin{lemma}
\label{lmag}
For any integer $n$ and $k\in (0,1)$,
\begin{align*}
\Gamma(kn)\leq((n-1)!)^{k},  \\
\Gamma(kn+1)\geq k^{n}(n!)^{k}.
\end{align*}
\end{lemma}
%-------------Proof
\begin{proof}
By logarithmic convexity,
\begin{align*}
\ln\Gamma(kn)&\leq\ln\Gamma(kn+1-k)<k\ln\Gamma(n)+(1-k)\ln\Gamma(1)=k\ln\Gamma(n) 
\end{align*}
Thus,
\begin{align*}
    \Gamma(kn) \leq\Gamma(n)^{k} = ((n-1)!)^k.  
\end{align*}
For the lower bound, we proceed by induction. When $n=1$, by Gautschi's inequality,
$$\Gamma(k+1)\geq \Gamma(k+(1-k))k^{1-(1-k)}=k^{k}\geq k =k^{1}(1!)^{k},$$ where we have used that facts that $k \in (0,1)$ and $\Gamma(1)=1$. The claim therefore holds for first step.
Let's assume it holds for step $n$, so we have  $\Gamma(kn+1)\geq k^{n}(n!)^{k}$. Again applying Gautschi's inequality, taking $k(n+1)$ for $x$ and $1-k$ for $s$, yields the lower bound for $\Gamma(k(n+1)+1)$: 
\begin{align*}\Gamma(k(n+1)+1)&\geq (k(n+1))^{k}\Gamma(k(n+1)+(1-k))=k^k(n+1)^{k}\Gamma(kn+1)\\&\geq k (n+1)^{k}(n!)^{k}k^{n}=k^{n+1}((n+1)!)^{k}.
\end{align*}
\end{proof}

%-------------Lemma: Inequality beta
\begin{lemma}
\label{inequality beta}
    When $n\in\N$, $0<\beta<1$, 
    \begin{align*}
    \left|\sum_{k=1}^{n}a_{k}\right|^{\beta}\leq \sum_{k=1}^{n}|a_{k}|^{\beta}.
    \end{align*}
\end{lemma}    
%-------------Proof
\begin{proof}
It clearly suffices to assume that each $a_k$ is positive. If $f(x) = 1 + x^\beta -(1+x)^\beta$, then it is easy to see that $f'(x) > 0$ whenever $x>0$. Since $f(0)=0$, we have $1 + x^\beta -(1+x)^\beta > 0$ for $x > 0$. Replace $x$ by $\frac{b}{a}$, and multiply through by $a^\beta$ to obtain
    \begin{align}
    \label{concave}
    (a+b)^{\beta}\leq a^{\beta}+b^{\beta}.
    \end{align}
This method combined with an easy induction gives the general result.
\end{proof}
\begin{lemma}
    \label{index function inequality}
    When $a>1$ and $0<k<1$,
    \begin{align*}
    a^{k}>ak.
    \end{align*}
    \begin{proof}
        Consider function $f(x)=a^{x}-ax$, then it is easy to see that $f'(x)=a^{x-1}x-a<0$ for $0<x<1$. As such $f(x)>f(1)=0$ for $0<x<1$. Therefore $a^{k}>ak$ when $a>1$ and $0<k<1$.
    \end{proof}
\end{lemma}
%-------------Lemma:
\begin{lemma}[{\citep[Lemma ~4.5]{hu2015stochastic}}]
\label{Beta function}
Let $\alpha\in (-1+\epsilon,1)^{m}$ with $\epsilon>0$ and set $|\alpha|=\sum_{i=1}^{m}\alpha_{i}$. $T_{m}(t)=\{(r_{1},r_{2},...,r_{m})\in \R^{m}:0<r_{1}<...<r_{m}<t\}$. Then there is a constant $c$ such that 
$$J_{m}(t,\alpha)=\int_{T_{m}(t)}\prod_{i=1}^{m}(r_{i}-r_{i-1})^{\alpha_{i}}dr\leq \frac{c^{m}t^{|\alpha|+m}}{\Gamma(|\alpha|+m+1)},$$
where by convention, $r_{0}=0$.
\end{lemma}    

%-------------Lemma
\begin{lemma}[{\citep[Proposition ~3.5.2]{garsia1970topics}}]
\label{Moments inequality}
Let $X_{1}, X_{2},..., X_{n}$ be independent, zero mean, random variables and suppose they belong to $L^{2p}$ for $p$ integer greater than 1.  Set $M=max_{1\leq v\leq n} \E [X_{v}^{2p}]^{\frac{1}{2p}}$. Then for all $(a_{1}, a_{2},..., a_{n}) \in \R^n$ we have 
$$\E \left [|a_{1}X_{1}+\cdots+a_{n}X_{n}|^{2p} \right ] \leq \frac{2^{p}(2p)!}{p!}M^{2p}(a_{1}^{2}+\cdots+a_{n}^{2})^{p}.$$
\end{lemma}
\begin{corollary}
\label{Moments inequality authentic}
    Let Let $X_{1}, X_{2},..., X_{n}$ be independent, zero mean, random variables and suppose they belong to $L^{p}$ for $p$ integer greater than 1.  Set $M=max_{1\leq v\leq n} \E [X_{v}^{p}]^{\frac{1}{p}}$. Then for all $(a_{1}, a_{2},..., a_{n}) \in \R^n$ we have 
$$\E \left [|a_{1}X_{1}+\cdots+a_{n}X_{n}|^{p} \right ]\leq 2^{\frac{5}{2}p+1}(p!)^{\frac{1}{2}}M^{p}(a_{1}^{2}+\cdots+a_{n}^{2})^{\frac{p}{2}}.$$
\begin{proof}
     We note \cref{Moments inequality} has established the even case, thus for even $p$ we can apply \cref{lmag} to $(\frac{p}{2})!$ and obtain 
     \begin{align}
     \label{authentic 1}
         \E \left [|a_{1}X_{1}+\cdots+a_{n}X_{n}|^{p} \right ]&\leq \frac{2^{\frac{p}{2}}(p)!}{(\frac{p}{2})!}M^{p}(a_{1}^{2}+\cdots+a_{n}^{2})^{\frac{p}{2}}\nonumber\\
         &\leq  \frac{2^{\frac{p}{2}}(p)!}{(\frac{1}{2})^{p}(p!)^{\frac{1}{2}}}M^{p}(a_{1}^{2}+\cdots+a_{n}^{2})^{\frac{p}{2}}\nonumber\\
         &\leq  2^{\frac{3p}{2}}((p)!)^{\frac{1}{2}}M^{p}(a_{1}^{2}+\cdots+a_{n}^{2})^{\frac{p}{2}}.
    \end{align}
     To prove the odd case, let $X = |a_{1}X_{1}+\cdots+a_{n}X_{n}|$ and assume $p$ is odd. Utilising Jensen's inequality, we obtain
    \begin{align*}
        \E[X^{p}]\leq \E[X^{p+1}]^{\frac{p}{p+1}}.
    \end{align*}
    Since $p+1$ is even, we can apply \cref{Moments inequality} to $\E[X^{p+1}]^{\frac{p}{p+1}}$ which yields
    \begin{align*}
     \E[X^{p+1}]^{\frac{p}{p+1}}&\leq \left(\frac{(p+1)!2^{\frac{p+1}{2}}}{(\frac{p+1}{2})!}M^{p+1}(a_{1}^{2}+\cdots+a_{n}^{2})^{\frac{p+1}{2}}\right)^{\frac{p}{p+1}}\\
     &\leq \frac{((p+1)!)^{\frac{p}{p+1}}2^{\frac{p}{2}}}{((\frac{p+1}{2})!)^{\frac{p}{p+1}}}M^{p}(a_{1}^{2}+\cdots+a_{n}^{2})^{\frac{p}{2}}.
    \end{align*}
     Applying \cref{lmag}, we can obtain an upper bound for $((p+1)!)^{\frac{p}{p+1}}$ and a lower bound for $(\frac{p+1}{2})!$ as follows:
     \begin{align*}
         ((p+1)!)^{\frac{p}{p+1}}\leq (p!)\left(\frac{p+1}{p}\right)^{p+1},\\
         \left(\frac{p+1}{2}\right)!\geq \left(\frac{1}{2}\right)^{p+1}((p+1)!)^{\frac{1}{2}}.
     \end{align*}
     Combining the aforementioned inequalities we can obtain an upper bound for $\E[X^{p+1}]^{\frac{p}{p+1}}$,
     \begin{align}
     \label{authentic 2}
         \E[X^{p+1}]^{\frac{p}{p+1}}&\leq \frac{2^{\frac{p}{2}}(p!)(\frac{p+1}{p})^{p+1}}{(\frac{1}{2})^{p}((p+1)!)^{\frac{p}{2(p+1)}}}M^{p}(a_{1}^{2}+\cdots+a_{n}^{2})^{\frac{p}{2}}\nonumber\\
         &\leq \frac{2^{p+\frac{p}{2}}2^{p+1}(p!)}{((p+1)!)^{\frac{p}{2(p+1)}}}M^{p}(a_{1}^{2}+\cdots+a_{n}^{2})^{\frac{p}{2}}\nonumber\\
         &\leq \frac{2^{\frac{5p}{2}+1}(p!)}{((p+1)!)^{\frac{p}{2(p+1)}}}M^{p}(a_{1}^{2}+\cdots+a_{n}^{2})^{\frac{p}{2}}\nonumber\\
         &\leq \frac{2^{\frac{5p}{2}+1}(p!)}{(p!)^{\frac{1}{2}}}M^{p}(a_{1}^{2}+\cdots+a_{n}^{2})^{\frac{p}{2}}\nonumber\\
         &\leq 2^{\frac{5p}{2}+1}(p!)^{\frac{1}{2}}M^{p}(a_{1}^{2}+\cdots+a_{n}^{2})^{\frac{p}{2}}.
     \end{align}
     The fourth inequality comes from the fact that $((p+1)!)^{\frac{p}{(p+1)}}\geq p!$ which for convenience we will prove here by induction. When $p=1$, $2^{\frac{1}{2}}\geq 1$ and the base case is verified. Assume $((p+1)!)^{\frac{p}{(p+1)}}\geq p!$ holds and we want to show
    \begin{align*}
        ((p+2)!)^{\frac{p+1}{(p+2)}}\geq (p+1)!.
    \end{align*}
    By assumption,
    \begin{align*}
        ((p+2)!)^{\frac{p+1}{p+2}}= (p+2)^{\frac{p+1}{p+2}}((p+1)!)^{\frac{p+1}{p+2}}&\geq (p+2)^{\frac{p+1}{p+2}}p!.
    \end{align*}
    Utilising \cref{index function inequality} we obtain,
    \begin{align*}
        (p+2)^{\frac{p+1}{p+2}}p!\geq (p+1)p!=(p+1)!,
    \end{align*}
    as required. Combining \cref{authentic 1} and \cref{authentic 2} obtains the wanted result.
\end{proof}
\end{corollary}

%----------------------------------------------------------Le Gall scheme--------------------------------------------------------
%--------------------------------------------------------------------------------------------------------------------------------

\section{Le Gall's scheme}
\label{appen:LeGall}
\noindent
%In order to study existence and regularity of self-intersection local-time of a process $Y$, Le Gall showed that one can approach this via the intersection local time of $Y$ and an independent copy $\tilde Y$. Using his scheme, one can show existence of self-intersection local-time by a partitioning argument. 
In this appendix, the scheme of Le Gall \citep{le2006temps, le1994exponential} will be briefly introduced. This scheme is crucial for studying SLT, as it shows that in order to study SLT, it is sufficient to study ILT, albeit under some restrictions. Moreover, these arguments can be adapted to the derivative case, namely one can show that to study DSLT, it is sufficient to consider DILT. For the purposes of illustration and for simplicity, we will consider the Le Gall scheme in the context of Brownian motion; however it also applies in the case of symmetric $\alpha$-stable processes. The following figure will illustrate the idea.
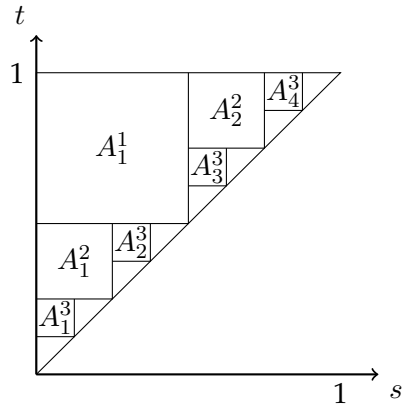
\begin{center}
\begin{tikzpicture}
\draw[thick,->] (0,0) -- (4.5,0) node[anchor=north west] {$s$};
\draw[thick,->] (0,0) -- (0,4.5) node[anchor=south east] {$t$};
\draw (0,0)--(0,4)--(4,4)--(0,0);
\draw  (0,2)--(2,2)--(2,4);
\draw (0,1)--(1,1)--(1,2);
\draw (2,3)--(3,3)--(3,4);
\draw (0,0.5)--(0.5,0.5)--(0.5,1);
\draw (1,1.5)--(1.5,1.5)--(1.5,2);
\draw (2,2.5)--(2.5,2.5)--(2.5,3);
\draw (3,3.5)--(3.5,3.5)--(3.5,4);
\draw (1,3) node{$A_{1}^{1}$};
\draw (0.5,1.5) node{$A_{1}^{2}$};
\draw (2.5,3.5) node{$A_{2}^{2}$};
\draw (0.25,0.75) node{$A_{1}^{3}$};
\draw (1.25,1.75) node{$A_{2}^{3}$};
\draw (2.25,2.75) node{$A_{3}^{3}$};
\draw (3.25,3.75) node{$A_{4}^{3}$};
\draw (-0.25,4) node{1};
\draw (4,-0.25) node{1};
\end{tikzpicture}
\end{center}
To be precise, for $n \in \N$ and $k \in \{1,2,\dots, 2^{n-1}\}$ the squares are given by
\begin{align*}
      A_{k}^{n}&= \big[(2(k-1)2^{-n},(2k-1)2^{-n} \big )\times \big ((2k-1)2^{-n},2k2^{-n})\big ].
\end{align*}
%where $n \in \N$ and $k \in \{1,2,\dots, 2^{n-1}\}$. 
In addition, we write 
\begin{align}
    L & :=\int_{0}^{1}\int_{0}^{1}\delta^{'}\left(B_{t}^{1}-B_s^{2}\right)dsdt \label{eqn:DILT} \\
            & = \frac{i}{2\pi}\int_0^1 \int_0^1 \int_{\R}pe^{ip(B^1_{t}-B^2_{s})}dpdsdt, \nonumber \\
     \hat{L}_{n,k} &:= \int_{A_{k}^{n}}\delta^{'}(B_t-B_s)dsdt, \label{eqn:DILTAkn} \\
	&=\frac{i}{2\pi}\int_{(2k-1)2^{-n}}^{2k2^{-n}}\int_{2(k-1)2^{-n}}^{(2k-1)2^{-n}}\int_{\R}pe^{ip(B_{t}-B_{s})}dpdsdt.  \nonumber 
\end{align}
The object \cref{eqn:DILT} is DILT whereas the object \cref{eqn:DILTAkn} is more or less DILT as the integration is done over $A_k^n$ (see the proof of \Cref{prop:LeGall} for more clarity). As they are written, the preceding objects seem like formal expressions. However, they can be well-defined by modifying the arguments utilised to define ILT in \citep{geman1984local}.

The original purpose of the Le Gall scheme was to utilise ILT in a clever way to define SLT. Here we will show that the scheme can be adapted to the derivative case. That is, we may use the previous objects given by \cref{eqn:DILT} and \cref{eqn:DILTAkn} in order to define the following:
\begin{align*}
     \hat{L} &= \int_0^1 \int_0^t \delta'(B_t - B_s) dsdt \\
                &= \frac{i}{2\pi} \int_{0}^{1}\int_{0}^{t}\int_{\R} pe^{ip(B_{t}-B_{s})}dpdsdt
\end{align*}
which we note is a formal expression for DSLT. This was originally considered by Rosen in \citep{rosen2005derivatives}, following upon related work by Rogers and Walsh \citep{rogers1991intrinsic, rogers1991local, rogers1991t}. The following proposition summarises the adaptation of the Le Gall scheme to the derivative case.
\begin{prop}
\label{prop:LeGall}
With the notation as above, $\hat{\alpha}$ and $\hat{\alpha}_{n,k}$ possess the following properties.
\begin{enumerate}[label = (\arabic*), ref = \arabic*]
    \item For each fixed $n \in \N$, $\hat{L}_{n,k}$ are mutually independent for $k\in \{1, 2, \dots, 2^{n-1}\}.$ \label{prop:1}
    \item $\hat{L}_{n,k} \overset{d}{=}2^{-n}L.$ \label{prop:2}
    \item $\hat{L}:=\sum_{n=1}^{\infty}\sum_{k=1}^{2^{n-1}}\hat{L}_{n,k}$, with convergence holding in all $L^p$ spaces and a.s. \label{prop:3}
\end{enumerate}
\end{prop}
\begin{proof}
%------------------------------------------------Independence
In order to prove \cref{prop:1}, it is enough to recognise that each rectangle $A_{k}^{n}$ does not overlap and that Brownian motion has independent increments. The mutual independence of $\hat{L}_{n,k}$ in $k\in \{1, 2, \dots, 2^{n-1}\}$ for each fixed $n \in \N$ then follows.

%------------------------------------------------Scaling 
In order to show \cref{prop:2}, we will change variables, let $2^{-n}u=t,2^{-n}v=s$. Then
\begin{align*}
	\hat{L}_{n,k}&=2^{-2n}\frac{i}{2\pi}\int_{2k-1}^{2k}\int_{2k-2}^{2k-1}\int_{\R}pe^{2^{-\frac{n}{2}}ip(B_{u}-B_{v})}dpdvdu \\
	&=2^{-2n}\frac{i}{2\pi}\int_{2k-1}^{2k}\int_{2k-2}^{2k-1}\int_{\R}pe^{2^{-\frac{n}{2}}ip(B_{u}-B_{2k-1}+B_{2k-1}-B_{v})}dpdvdu \\
	&\overset{d}{=}2^{-2n}\frac{i}{2\pi}\int_{2k-1}^{2k}\int_{2k-2}^{2k-1}\int_{\R}pe^{2^{-\frac{n}{2}}ip(B_{u-2k+1}^{1}-B_{2k-1-v}^{2})}dpdvdu.
\end{align*}
Here $B^{1}$ and $B^{2}$ refer to two independent Brownian motions, and since $B_{u}-B_{2k-1}$ and $B_{2k-1}-B_{v}$ are independent the last equality in distribution is legitimate. We now change variable again, letting $w=u-2k+1, r=2k-1-v$, hence
\begin{align*}
    	\hat{L}_{n,k}&\overset{d}{=}-2^{-2n}\frac{i}{2\pi}\int_{0}^{1}\int_{1}^{0}\int_{\R}pe^{2^{-\frac{n}{2}}ip(B_{r}^{1}-B_{w}^{2})}dpdrdw \\
   	 &\overset{d}{=}2^{-2n}\frac{i}{2\pi}\int_{0}^{1}\int_{0}^{1}\int_{\R}pe^{2^{-\frac{n}{2}}ip(B_{r}^{1}-B_{w}^{2})}dpdrdw.
\end{align*}
Changing variable for $p$ by letting $p=2^{\frac{n}{2}}\eta$ yields
\begin{align*}
    	\hat{L}_{n,k}&\overset{d}{=}2^{-n}\frac{i}{2\pi}\int_{0}^{1}\int_{0}^{1}\int_{\R}\eta e^{i\eta(B_{r}^{1}-B_{w}^{2})}d\eta drdw \\
    	&\overset{d}{=}2^{-n}\int_{0}^{1}\int_{0}^{1}\delta^{'}\left(B_{r}^{1}-B_w^{2}\right)drdw\\
     &\overset{d}{=}2^{-n}L.
\end{align*}
Consequently \cref{prop:2} is shown.

\cref{prop:3} requires more care than the other two parts, and is shown in detail in \citep{rosen2005derivatives}. The main idea is to exploit independence by using \Cref{Moments inequality} to provide strong estimates for the terms in the sum in question.
\end{proof}

%-----------------------------------------------------------References-----------------------------------------------------------
%--------------------------------------------------------------------------------------------------------------------------------

\bibliographystyle{plainurl}
\bibliography{citation_2}

\end{document}